\documentclass[a4, 12pt]{amsart}
\usepackage[mathscr]{eucal}
\usepackage{amssymb}
\usepackage{latexsym}
\usepackage{amsthm}
\usepackage{color}
\theoremstyle{plain}
\newtheorem{theorem}{Theorem}[section]

\newtheorem{lemma}{Lemma}[section]

\usepackage{ascmac}
\usepackage{mathrsfs}

\makeatletter
\@addtoreset{equation}{section}

\title[Complete minimal hypersurfaces]
{Complete minimal hypersurfaces \\ in a hyperbolic space $\mathbf H^4(-1)$}
\author{Qing-Ming Cheng and Yejuan Peng}

\address{Qing-Ming Cheng \\ Department of Applied Mathematics, \newline
\indent Faculty of Science, Fukuoka  University, \newline
\indent 814-0180, Fukuoka,  Japan, cheng@fukuoka-u.ac.jp}
\address{Yejuan Peng\\ College of Mathematics and Information Sciences, \newline
\indent Henan Normal  University, 453007, Xinxiang, China, yejuan666@126.com }
\begin{document}
\maketitle

\begin{abstract}In this paper, we study $n$-dimensional complete
minimal hypersurfaces in a hyperbolic space $H^{n+1}(-1)$ of constant curvature $-1$.  We prove that a
$3$-dimensional complete  minimal hypersurface with constant scalar
curvature in $H^{4}(-1)$  satisfies $S\leq \frac{21}{29}$ by making use of the Generalized Maximum Principle,
where $S$ denotes the
squared norm of the second fundamental form of the hypersurface.
\end{abstract}

\section{Introduction}
Let $M^n$ be an $n$-dimensional minimal hypersurface in the hyperbolic space $H^{n+1}(-1)$ of
constant curvature $-1$.  It is a very important subject to study the rigidity of complete minimal hypersurfaces
in  the hyperbolic space $H^{n+1}(-1)$. It is well-known that there are many important results on the rigidity
of compact minimal hypersurfaces in the unit sphere $S^{n+1}(1)$. For examples, Simons \cite{s}, Chern-do Carmo-
Kobayashi \cite{cdk} and Lawson \cite{l} prove that an $n$-dimensional compact minimal hypersurface  in the unit sphere $S^{n+1}(1)$
is isometric to a totally geodesic sphere or a Clifford torus if  the squared norm $S$
of its  the second fundamental form satisfies $S\leq n$. In particular, for $n=3$, it is known that a $3$-dimensional compact
minimal hypersurface  in the unit sphere $S^{4}(1)$ with constant scalar curvature
is isometric to a totally geodesic sphere or a Clifford torus  or the Cartan minimal isoparametric hypersurface (cf. \cite{pt}, \cite{c}).
On the other  hand, Cheng and Wan \cite{cw} proved  complete minimal hypersufraces with constant scalar curvature in the Euclidean space $\mathbb R^4$
is isometric to the  hyperplane $\mathbb R^3$.  But for complete minimal hypersurfaces in the hyperbolic space $H^{n+1}(-1)$ , there are only a little results
on rigidity of complete minimal hypersurfaces. It is our main purpose to study the following conjecture:

\vskip2mm
\noindent
{\bf Conjecture.}   A complete minimal hypersurfaces with constant scalar curvature in  the hyperbolic space $H^{4}(-1)$
is  isometric to   the hyperbolic space $H^{3}(-1)$.

We will prove the following:
\begin{theorem}
A complete minimal hypersurfaces with constant scalar curvature in  the hyperbolic space $H^{4}(-1)$
satisfies $S\leq \dfrac{21}{29}$, where $S$ denotes the
squared norm of the second fundamental form of the hypersurface.
\end{theorem}
\newpage
\section{Basic formulas}
Let $M^{n}$ be an n-dimensional hypersurface in an $(n+1)$-dimensional hyperbolic space $H^{n+1}(-1)$.
At each point $p$ in $H^{n+1}(-1)$, we choose a local orthonormal frame field
$\{e_{1},e_{2},\cdots.e_{n+1}\}$ and the dual
coframe
$\{\omega^{1},\omega^{2},\cdots,\omega^{n+1}\}$
such that, restricted to $M^n$,  $\{e_{1},e_{2},\cdots,e_{n}\}$ is tangent to $M^{n}$ .  Then the structure eqquations of $N^{n+1}$ are given by
\begin{equation}
\begin{split}
d\omega_{A}=-\sum_{B}\omega_{AB}\land\omega_{B},\ \ \omega_{AB}+\omega_{BA}=0,\\
d\omega_{AB}+\sum_{C}\omega_{AC}\land\omega_{CB}=\frac{1}{2}\sum_{C,D}K_{ABCD}\ \omega_{C}\land\omega_{D}.
\end{split}
\end{equation}
with
\begin{equation}
\begin{split}
&K_{ABCD}=-(\delta_{AC}\delta_{BD}-\delta_{AD}\delta_{BC}).
\end{split}
\end{equation}
If we restrict these forms to $M^{n}$, then $\omega^{n+1}$=0. We have
\begin{equation}
\omega_{n+1, i}=\sum_{j}h_{ij}\omega_{j},\ \ \ \ \  h_{ij}=h_{ji}.
\end{equation}
One  calls
\begin{equation}
H=\frac{1}{n}\sum_{i}h_{ii}, \  \
h=\sum_{i,j}h_{ij}\omega_{i}\otimes\omega_{j}
\end{equation}
are the mean curvature and the second fundamental form of $M^{n}$ respectively. If $H$ is identically zero,  $M^{n}$ is called  to be  minimal.
The structure equations of $M^n$ are given by
\begin{equation}
\begin{split}
d\omega_{i}=-\sum_{j}\omega_{ij}\land\omega_{j},\ \ \omega_{ij}+\omega_{ji}=0,\\
d\omega_{ij}+\sum_{k}\omega_{ik}\land\omega_{kj}=\frac{1}{2}\sum_{k,l}R_{ijkl}\ \omega_{k}\land\omega_{l},
\end{split}
\end{equation}
where
\begin{equation}
R_{ijkl}= -(\delta_{ik}\delta_{jl}-\delta_{il}\delta_{jk})+(h_{ik}h_{jl}-h_{il}h_{jk}).
\end{equation}
For minimal hypersurfaces in $H^{n+1}(-1)$, we obtain
\begin{align*}
  R=-n(n-1)-S,
\end{align*}
where $R$ and $S$ denote the scalar curvature and the squared norm
of the second fundamental form of $M^n$, respectively. From the structure equations of $M^n$,
Codazzi equations and  Ricci formulas are given by
\begin{align*}
h_{ijk}=h_{ikj},
\end{align*}
\begin{align*}
h_{ijkl}-h_{ijlk}=\sum_mh_{im}R_{mjkl}+\sum_mh_{mj}R_{mikl}.
\end{align*}
where  $h_{ijk}=\nabla_kh_{ij}$ and $h_{ijkl}=\nabla_l\nabla_kh_{ij} $,
respectively.

We define functions $f_3$ and $f_4$ by
\begin{align*}
f_3=\sum_{i,j,k=1}^nh_{ij}h_{jk}h_{ki} \ \ {\rm and} \ \
f_4=\sum_{i,j,k,l=1}^nh_{ij}h_{jk}h_{kl}h_{li}
\end{align*}
respectively.
 Then, we have, for minimal hypersurfaces,
  \begin{equation}
  \dfrac13\Delta f_3=-(n+S)f_3+2C,
 \end{equation}
\begin{equation}
  \dfrac14\Delta f_4= -(n+S)f_4+ (2A+B),
 \end{equation}
 where
 $$C=\sum_{i,j,k} \lambda_ih_{ijk}^2, \  \  A=\sum_{i,j,k}\lambda_i^2h_{ijk}^2, \ \
B=\sum_{i,j,k}\lambda_i\lambda_jh_{ijk}^2
$$
and $\lambda_i$'s are principal curvatures of $M^n$.
\begin{align*}
&\sum_ih_{ii}=\sum_i\lambda_i=0, \quad
S=\sum_{i,j}h_{ij}^2=\sum_i\lambda_i^2, \\
&h_{ijij}-h_{jiji}=(\lambda_i-\lambda_j)(-1+\lambda_i\lambda_j).
\end{align*}
By a direct computation, we have
\begin{equation*}
\begin{split}
&S=n(1-n)-R, \\
&\Delta h_{ij}=-(S+n)h_{ij}.\\
&\frac{1}{2}\Delta S=-S(S+n)+\sum_{i,j,k}h_{ijk}^{2}.\\
\end{split}
\end{equation*}

If the squared norm $S$ of the second fundamental
form is constant, we have
\begin{equation*}
\begin{split}
&\sum_{i,j,k}h_{ijk}^{2}=S(S+n).\\
&\sum_{i,j,k,l}h_{ijkl}^{2}=S(S+n)(2n+3+S)+3(A-2B).\\
\end{split}
\end{equation*}

The following Generalized Maximum Principle due to Omori \cite{o} (cf. Yau \cite{y}) will play an important role in this paper.
\begin{theorem}
Let $M^{n}$ be a complete Riemannian manifold with sectional  curvature bounded from below.
If a $\mathcal C^2$-function $f$ is bounded from above in $M^{n}$, then there exists a sequence of $\{p_{k}\}_{k=1}^{\infty} \subset{M^{n}}$ such that
\begin{enumerate}
\item $\lim_{k \to \infty} f(p_{k})=\underset{M^{n}}{\sup}\ f$
\item $\lim_{k \to \infty}|\nabla f(p_{k})|=0$
\item $\lim_{k \to \infty}\sup\ \nabla_l\nabla_lf(p_{k})\leq 0$, for  $l=1,  2,  \cdots, n$.
\end{enumerate}
\end{theorem}

\section{Minimal hypersurfaces with two distinct principlal curvatures}

\begin{theorem}
Let $M^{3}$ be a minimal  hypersurface  in $H^{4}(-1)$ with constant scalar curvature. If $M^{3}$ has two principal curvatures somewhere,
 we have $S\leq \frac{21}{29}$.
\end{theorem}

\vskip2mm
\noindent
{\it Proof.}
We assume,  at $p\in M^{3}$,  that $M^{3}$ has  two distinct principal curvatures.
At $p$,  we may choose  an orthonormal frame   $e_1, e_2, e_3$ such that $h_{ij}=\lambda_{i}\delta_{ij}$.
We can let
\begin{equation*}
\lambda_{1}=\lambda_{2}=\lambda.
\end{equation*}
Since $M^{3}$  is minimal, we have
\begin{equation*}
\lambda_{3}=-2\lambda,\  \  \lambda^{2}=\frac{S}{6}.
\end{equation*}
Since $\sum_{i}h_{ii}=0$ and $S$ is constant, we have
\begin{equation*}
h_{11k}+h_{22k}+h_{33k}=0,\
h_{11k}+h_{22k}-2h_{33k}=0.
\end{equation*}
We obtain
\begin{equation*}
h_{11k}+h_{22k}=0, \ h_{33k}=0,\ k=1,2,3.
\end{equation*}
We can choose $e_{1},e_{2}$ such that ${h}_{123}(p)=0$ at $p$. In fact, if necessary, we make a rotation of
$e_{1},e_{2}$  with angle $\theta$, which satisfies
\begin{equation*}
\cos(-2\theta)=\frac{h_{223}(p)}{\sqrt{h_{223}^{2}(p)+h_{123}^{2}(p)}}, \
\sin(-2\theta)=\frac{h_{123}(p)}{\sqrt{h_{223}^{2}(p)+h_{123}^{2}(p)}}.
\end{equation*}
Letting
\begin{equation*}
a=h_{113}^{2}, \ \ b=h_{111}^{2}+h_{112}^{2},
\end{equation*}
since
\begin{equation*}
 \begin{split}
&S(S+3)=
\sum_{i,j,k}h_{ijk}^{2}\\
&=3(h_{112}^{2}+h_{113}^{2}+h_{221}^{2}+h_{223}^{2})+(h_{111}^{2}+h_{222}^{2})\\
&=6h_{113}^{2}+4(h_{111}^{2}+h_{112}^{2}),
\end{split}
\end{equation*}
we have
\begin{equation*}
6a+4b=S(S+3),
\end{equation*}
Since $n=3$, we have
$$
f_{4}=\frac{1}{2}S^{2},  \  \   2A+B=\frac{1}{2}S^{2}(S+3).\
$$
\begin{lemma}
$h_{ijkl}$ are symmetric in $i, j, k, l$ if $i, j, k, l$ are  not $\{1,1,3,3\}, \{2,2,3,3\}$  and
\begin{equation*}
\begin{split}
&h_{3311}=h_{3322}=\frac{2}{3\lambda}{(a+b)}, \
h_{3333}=\frac{2a}{3\lambda}, \
h_{3312}=0, \
h_{3313}=\frac{2}{3\lambda}h_{111}h_{113}, \\
 &h_{3323}=\frac{2}{3\lambda}h_{112}h_{113},\
h_{1111}=h_{2222}, \
h_{1133}=h_{2233}=-\frac{a}{3\lambda}.
\end{split}
\end{equation*}
\end{lemma}

\begin{proof} According to the Ricci formula,
\begin{equation*}
\begin{split}
h_{ijkl}-h_{ijlk}
&=\sum_{m}h_{mj}R_{mikl}+\sum_{m}h_{im}R_{mjkl}\\
&=(\lambda_{i}-\lambda_{j})R_{ijkl}\\
&=(\lambda_{i}-\lambda_{j}(-1+\lambda_{i}\lambda_{j})(\delta_{ik}\delta_{jl}-\delta_{il}\delta_{jk}).
\end{split}
\end{equation*}
and  $S=\sum_{i,j}h_{ij}^{2}$ is constant, we have
\begin{equation*}
\begin{split}
0&=\sum_{i,j}(h_{ij}^{2})_{kl}\\
&=2(\sum_{i,j}h_{ijk}h_{ijl}+\sum_{i,j}h_{ij}h_{ijkl})\\
&=2(\sum_{i,j}h_{ijk}h_{ijl}-3\lambda h_{33kl}).
\end{split}
\end{equation*}
Thus, we get our conclusion.

\end{proof}
\begin{lemma}
\label{3.5.}
\begin{equation}
\label{3.1}
x+2y=\frac{26}{9}a^{2}+\frac{7}{18}ab-b^{2}+\frac{5}{4}Sb,
\end{equation}
where
\begin{equation*}
\begin{split}
x=\lambda^{2}[3(h_{1123}^{2}+h_{2213}^{2})+h_{1113}^{2}+h_{2223}^{2}],\\
y=\lambda^{2}(h_{1111}^{2}+h_{1112}^{2})+(a+b)\lambda h_{1111}.
\end{split}
\end{equation*}
\end{lemma}
\begin{proof}
\begin{equation*}
\begin{split}
\sum_{i,j,k,l}h_{ijkl}^{2}&=S(S+3)(S+9)+3(A-2B)\\
&=S(S+3)(S+9)+4(2A+B)-5(A+2B)\\
&=S(S+3)(S+9)+2S^{2}(S+3)-5(\sum_{i,j,k}h_{ijk}^{2}\lambda_{i}^{2}+2\sum_{i,j,k}h_{ijk}^{2}\lambda_{i}\lambda_{j})\\
&=3S(S+3)^{2}-\frac{5}{3}\sum_{i,j,k}h_{ijk}^{2}(\lambda_{i}+\lambda_{j}+\lambda_{k})^{2},
\end{split}
\end{equation*}
where
\begin{equation*}
\begin{split}
\sum_{i,j,k}h_{ijk}^{2}(\lambda_{i}+\lambda_{j}+\lambda_{k})^{2}&
=3\sum_{i\neq k}h_{iik}^{2}(2\lambda_{i}+\lambda_{k})^{2}+9\sum_{i}h_{iii}^{2}\lambda_{i}^{2}\\
&=36\lambda^{2}b.
\end{split}
\end{equation*}
Hence we have
\begin{equation*}
\begin{split}
\sum_{i,j,k,l}h_{ijkl}^{2}&=3S(S+3)^{2}-60\lambda^{2}b=3S(S+3)^{2}-10Sb.
\end{split}
\end{equation*}
\begin{equation*}
\begin{split}
&\sum_{i,j,k}h_{ijk1}^{2}
=\sum_{i\neq j\neq k}h_{ijk1}^{2}+3\sum_{i\neq k}h_{iik1}^{2}+\sum_{i}h_{iii1}^{2}\\
&=6h_{1231}^{2}+3(h_{1121}^{2}+h_{1131}^{2}+h_{2211}^{2}+h_{2231}^{2}+h_{3311}^{2})\\
&+(h_{1111}^{2}+h_{2221}^{2}+h_{3331}^{2})\\
&=3(2h_{1123}^{2}+h_{1113}^{2}+h_{2213}^{2})+(h_{1111}^{2}+3h_{1112}^{2})+h_{3331}^{2}\\
&+(h_{2221}^{2}+3h_{2211}^{2}+3h_{3311}^{2})\\
&=3(2h_{1123}^{2}+h_{1113}^{2}+h_{2213}^{2})+4(h_{1111}^{2}+h_{1112}^{2})\\
&+h_{3331}^{2}+6(h_{1111}h_{3311}+h_{3311}^{2}).
\end{split}
\end{equation*}
In the same way, we have
\begin{equation*}
\begin{split}
&\sum_{i,j,k}h_{ijk2}^{2}
=\sum_{i\neq j\neq k}h_{ijk2}^{2}+3\sum_{i\neq k}h_{iik2}^{2}+\sum_{i}h_{iii2}^{2}\\
&=3(2h_{2213}^{2}+h_{1123}^{2}+h_{2223}^{2})+4(h_{1111}^{2}+h_{1112}^{2})\\
&+h_{3332}^{2}+3(2h_{1111}h_{3322}+h_{3311}^{2}+h_{3322}^{2}).
\end{split}
\end{equation*}
\begin{equation*}
\begin{split}
&\sum_{i,j,k}h_{ijk3}^{2}=\sum_{i\neq j\neq k}h_{ijk3}^{2}+3\sum_{i\neq k}h_{iik3}^{2}+\sum_{i}h_{iii3}^{2}\\
&=3(h_{1123}^{2}+h_{2213}^{2})+h_{1113}^{2}+h_{2223}^{2}+3(h_{1133}^{2}+h_{2233}^{2}\\
&+h_{3313}^{2}+h_{3323}^{2})+h_{3333}^{2}\\
&=3(h_{1123}^{2}+h_{2213}^{2})+h_{1113}^{2}+h_{2223}^{2}+\bigg(\frac{2a^{2}}{3\lambda^{2}}+\frac{4}{3\lambda^{2}}ab\bigg)+\frac{4a^{2}}{9\lambda^{2}}\\
&=3(h_{1123}^{2}+h_{2213}^{2})+h_{1113}^{2}+h_{2223}^{2}+\frac{10a^{2}+12ab}{9\lambda^{2}}.
\end{split}
\end{equation*}
Hence we obtain
\begin{equation*}
\begin{split}
&\sum_{i,j,k,l}h_{ijkl}^{2}=\sum_{i,j,k}h_{ijk1}^{2}+\sum_{i,j,k}h_{ijk2}^{2}+\sum_{i,j,k}h_{ijk3}^{2}\\
&=12(h_{1123}^{2}+h_{2213}^{2})+4(h_{1113}^{2}+h_{2223}^{2})+8(h_{1111}^{2}+h_{1112}^{2})\\
&+12h_{1111}h_{3311}+(h_{3331}^{2}+h_{3332}^{2})
+12h_{3311}^{2}+\frac{10a^{2}+12ab}{9\lambda^{2}}\\
&=\bigg[12(h_{1123}^{2}+h_{2213}^{2})+4(h_{1113}^{2}+h_{2223}^{2})\bigg]\\
&+\bigg[8(h_{1111}^{2}+h_{1112}^{2})+\frac{8}{\lambda}h_{1111}(a+b)\bigg]\\
&+\frac{4}{9\lambda^{2}}ab+12\bigg[\frac{2}{3\lambda}(a+b)\bigg]^{2}+\frac{10a^{2}+12ab}{9\lambda^{2}}.
\end{split}
\end{equation*}
Hence, we infer from the above formulas,
\begin{equation*}
\begin{split}
&\frac{4}{\lambda^{2}}x+\frac{8}{\lambda^{2}}y+\frac{4ab+48(a+b)^{2}+10a^{2}+12ab}{9\lambda^{2}}=3S(S+3)^{2}-10Sb,
\end{split}
\end{equation*}
that is,
\begin{align*}
x+2y&=\frac{26}{9}a^{2}+\frac{26}{9}ab+\frac{2}{3}b^{2}-\frac{5}{12}S^2b\\
&=\frac{26}{9}a^{2}+\frac{7}{18}ab-b^{2}+\frac{5}{4}Sb.
\end{align*}
\end{proof}

Furthermore, we have
\begin{lemma}
\label{3.6.}
\begin{equation}
\label{3.2}
x+4a\lambda h_{1111}=-\frac{34}{9}a^{2}-\frac{4}{3}ab+\frac{4}{3}b^{2}+\lambda^{2}(72\lambda^{2}+18)a+\lambda^{2}(40\lambda^{2}+8)b.
\end{equation}
\end{lemma}
\begin{proof}
Since $S=\sum_{i,j}h_{ij}^{2}$ is constant, we get
\begin{equation*}
\begin{split}
0&=\sum_{i,j}(h_{ij}^{2})_{klm}\\
&=2\sum_{i,j}(h_{ij}h_{ijklm}+h_{ijm}h_{ijkl}+h_{ijk}h_{ijlm}+h_{ijl}h_{ijkm}).
\end{split}
\end{equation*}
Since
\begin{equation*}
\sum_{i,j}h_{ij}h_{ijklm}=-3\lambda h_{33klm},
\end{equation*}
we have
\begin{equation*}
3\lambda h_{33klm}=\sum_{i,j}h_{ijm}h_{ijkl}+\sum_{i,j}h_{ijk}h_{ijlm}+\sum_{i,j}h_{ijl}h_{ijkm}.
\end{equation*}
Hence,
\begin{equation*}
\sum_{k,l,m}h_{klm}h_{33klm}=\frac{1}{\lambda}\sum_{i,j,k,l,m}h_{ijk}h_{klm}h_{ijlm}.
\end{equation*}
On the other hand, we have
\begin{equation*}
\begin{split}
0&=\sum_{i,j,k}(h_{ijk}^{2})_{33}\\
&=2\sum_{i,j,k}(h_{ijk}h_{ijk33}+h_{ijk3}^{2}).
\end{split}
\end{equation*}

\begin{equation*}
\begin{split}
\sum_{i,j,k}h_{ijk}(h_{33ijk}-h_{ijk33})&=\sum_{i,j,k}h_{ijk}h_{33ijk}+\sum_{i,j,k}h_{ijk3}^{2}\\
&=\frac{1}{\lambda}\sum_{i,j,k,l,m}h_{ijk}h_{klm}h_{ijlm}+\frac{x}{\lambda^{2}}+\frac{10a^{2}+12ab}{9\lambda^{2}}.
\end{split}
\end{equation*}

Since
\begin{equation*}
\begin{split}
&\sum_{i,j,k}h_{ijk}(h_{33ijk}-h_{ijk33})\\
&=\sum_{i,j,k}h_{ijk}[h_{3i3jk}-h_{ijk33}]\\
&=\sum_{i,j,k}h_{ijk}\bigg[(h_{3ij3}+\sum_{m}h_{mi}R_{m33j}+\sum_{m}h_{3m}R_{mi3j})_{k}-(h_{ij3k}+2\sum_{m}h_{mj}R_{mik3})_{3}\bigg]\\
&=\sum_{i,j,k}h_{ijk}[h_{3ij3k}-h_{ij3k3}+\sum_{m}h_{mik}R_{m33j}+\sum_{m}h_{3mk}R_{mi3j}-2\sum_{m}h_{mj3}R_{mik3}]\\
&+\sum_{i,j,k,m}h_{ijk}h_{mi}(h_{m3}h_{3j}-h_{mj}h_{33})_{k}+\sum_{i,j,k,m}h_{ijk}h_{3m}(h_{m3}h_{ij}-h_{mj}h_{i3})_{k}\\
&-2\sum_{i,j,k,m}h_{ijk}h_{mj}(h_{mk}h_{i3}-h_{m3}h_{ik})_{3}\\
&=\sum_{i,j,k}h_{ijk}[2\sum_{m}h_{mij}R_{m33k}+5\sum_{m}h_{3mj}R_{mi3k}]\\
&+\sum_{i,j,k,m}h_{ijk}h_{mi}(h_{m3k}h_{3j}+h_{m3}h_{3jk}-h_{mjk}h_{33})\\
&+\sum_{i,j,k,m}h_{ijk}h_{3m}(h_{m3}h_{ijk}-h_{mjk}h_{i3}-h_{mj}h_{i3k}+h_{m3k}h_{ij})\\
&-2\sum_{i,j,k,m}h_{ijk}h_{mj}(h_{mk3}h_{i3}-h_{m3}h_{ik3})\\
&=\sum_{i,j,k,m}[2h_{ijk}h_{mij}(-1+\lambda_{k}\lambda_{3})(\delta_{k3}\delta_{3m}-\delta_{mk}\delta_{33})\\
&+5h_{ijk}h_{3mj}(-1+\lambda_{i}\lambda_{k})(\delta_{m3}\delta_{ik}-\delta_{mk}\delta_{i3})]
+\sum_{i,k}\lambda_{3}\lambda_{i}h_{i3k}^{2}-\sum_{i,j,k}\lambda_{3}\lambda_{i}h_{ijk}^{2}\\
&+\sum_{j,k}\lambda_{3}^{2}h_{3jk}^{2}+\sum_{i,j,k}h_{ijk}^2\lambda_3^2-\sum_{j,k}h_{3jk}^2\lambda_3^2
-\sum_{i,k}\lambda_{3}^{2}h_{i3k}^{2}
-2\sum_{j,k}\lambda_{3}\lambda_{j}h_{3jk}^{2}+2\sum_{i,k}\lambda_{3}^{2}h_{i3k}^{2}\\
&=[2\sum_{i,j}h_{ij3}^{2}(-1+\lambda_{3}^{2})-2\sum_{i,j,k}h_{ijk}^{2}(-1+\lambda_{k}\lambda_{3})
-5\sum_{j,k}h_{3jk}^{2}(-1+\lambda_{3}\lambda_{k})]\\
&-\sum_{i,j}\lambda_{3}\lambda_{i}h_{ij3}^{2}-\sum_{i,j,k}\lambda_{3}\lambda_{i}h_{ijk}^{2}
+\sum_{i,j,k}\lambda_{3}^{2}h_{ijk}^{2}+\sum_{i,j}\lambda_{3}^{2}h_{ij3}^{2}\\
&=[2(-1+4\lambda^{2})(2a)-2(-(6a+4b)-2\lambda(4\lambda b))+5(1+2\lambda^{2})(2a)]\\
&+4a\lambda^{2}+2\lambda(4\lambda b)+4\lambda^{2}(6a+4b)+8\lambda^{2}a\\
&=(72\lambda^{2}+18)a+(40\lambda^{2}+8)b.
\end{split}
\end{equation*}
\begin{equation*}
\begin{split}
&\sum_{i,j,k,l,m}h_{ijk}h_{klm}h_{ijlm}\\
&=\sum_{k,l,m}h_{klm}(h_{11k}h_{11lm}+h_{22k}h_{22lm}+2h_{12k}h_{12lm}+2h_{13k}h_{13lm}+2h_{23k}h_{23lm})\\
&=\sum_{k,l,m}h_{11k}h_{klm}(h_{11lm}-h_{22lm})+\sum_{l,m}2(h_{112}h_{1lm}-h_{111}h_{2lm})h_{12lm}\\
&+\sum_{l,m}2h_{113}h_{1lm}h_{13lm}-\sum_{l,m}2h_{113}h_{2lm}h_{23lm}\\
&=\sum_{k}h_{11k}[h_{k11}(h_{1111}-h_{2211})+h_{k22}(h_{1122}-h_{2222})+2h_{k13}(h_{1113}-h_{2213})\\
&-2h_{k23}(h_{1123}-h_{2223}))]+2h_{k12}(h_{1112}-h_{2212})\\
&+4h_{112}(h_{112}h_{1212}+h_{113}h_{1213}+h_{111}h_{1211}+h_{122}h_{1222})\\
&-4h_{111}(h_{212}h_{1212}+h_{211}h_{1211}+h_{222}h_{1222}+h_{223}h_{1223})\\
&+2h_{113}[h_{111}h_{1311}+h_{113}(h_{1313}+h_{1331})+2h_{112}h_{1312}+h_{122}h_{1322}]\\
&-2h_{113}[h_{222}h_{2322}+2h_{212}h_{2312}+h_{223}(h_{2323}+h_{2332})+h_{211}h_{2311}]\\
&=(a+b)(h_{1111}-h_{2211})+\sum_{k}h_{11k}^{2}(h_{1111}-h_{2211})\\
&+4bh_{1122}+2a(h_{1133}+h_{3311}+h_{2233}+h_{3322})\\
&+4h_{111}h_{113}h_{1113}-4h_{112}h_{223}h_{2223}+4h_{112}h_{113}h_{1123}-4h_{113}h_{221}h_{2213}\\
&=2(a+b)(h_{1111}-h_{2211})+4bh_{1122}\\
&+4h_{113}(h_{111}h_{1113}+h_{112}h_{2223}+h_{112}h_{1123}+h_{111}h_{2213})\\
&+2a(h_{1133}+h_{3311}+h_{2233}+h_{3322})\\
&=2(a+b)(h_{1111}-h_{2211})+4bh_{1122}\\
&-4h_{113}(h_{111}h_{3313}+h_{112}h_{3323})+2a(h_{1133}+h_{3311}+h_{2233}+h_{3322})\\
&=2(a+b)(2h_{1111}+h_{3311})-4b(h_{1111}+h_{3311})\\
&-4h_{113}(h_{111}\cdot \frac{2}{3\lambda}h_{111}h_{113}+h_{112}\cdot \frac{2}{3\lambda}h_{112}h_{113})\\
&+2a\bigg(-\frac{a}{3\lambda}+\frac{2}{3\lambda}(a+b)-\frac{a}{3\lambda}+\frac{2}{3\lambda}(a+b)\bigg)\\
&=4ah_{1111}+2(a-b)\cdot\frac{2}{3\lambda}(a+b)
-\frac{8}{3\lambda}a(h_{111}^{2}+h_{112}^{2})
-\frac{4a^{2}}{3\lambda}+\frac{8a}{3\lambda}(a+b)\\
&=4ah_{1111}+\frac{8a^{2}-4b^{2}}{3\lambda}.
\end{split}
\end{equation*}
Hence, we have
\begin{equation*}
\begin{split}
(72\lambda^{2}+18)a+(40\lambda^{2}+8)b
=&\frac{4ah_{1111}}{\lambda}+\frac{8a^{2}-4b^{2}}{3\lambda^{2}}+\frac{x}{\lambda^{2}}+\frac{10a^{2}+12ab}{9\lambda^{2}}\\
x+4a\lambda h_{1111}=&\lambda^{2}(72\lambda^{2}+18)a+\lambda^{2}(40\lambda^{2}+8)b-\frac{34a^{2}+12ab-12b^{2}}{9}\\
=&-\frac{34}{9}a^2-\frac{4}{3}ab+\frac{4}{3}b^{2}+\lambda^{2}(72\lambda^{2}+18)a+\lambda^{2}(40\lambda^{2}+8)b
\end{split}
\end{equation*}
\end{proof}
Since equation $(\ref{3.1})$ and $(\ref{3.2})$, we have from $6a+4b=S(S+3)$, $6\lambda^2=S$
\begin{equation}
\label{3.3}
\begin{aligned}
&2\lambda^{2}(h_{1111}^{2}+h_{1112}^{2})-2(a-b)\lambda h_{1111}\\
&=\frac{60}{9}a^{2}+\frac{31}{18}ab-\frac{7}{3}b^{2}+\frac{5}{4}Sb-\lambda^{2}(72\lambda^{2}+18)a-\lambda^{2}(40\lambda^{2}+8)b\\
&=\frac{20}{3}a^{2}+\frac{31}{18}ab-\frac{7}{3}b^{2}+\frac{5}{4}Sb-\frac{S}{6}(12S+18)a-\frac{S}{6}(\frac{20}{3}S+8)b\\
&=\frac{20}{3}a^{2}+\frac{31}{18}ab-\frac{7}{3}b^{2}-2S(S+3)a+3Sa-\frac{10}{9}S(S+\frac{3}{40})b\\
&=\frac{20}{3}a^{2}+\frac{31}{18}ab-\frac{7}{3}b^{2}-2(6a+4b)a+3Sa-\frac{10}{9}S(S+3-\frac{117}{40})b\\
&=-\frac{16}{3}a^{2}-\frac{233}{18}ab-\frac{61}{9}b^{2}+S(3a+\frac{13}{4}b),
\end{aligned}
\end{equation}
According to
\begin{equation*}
\begin{split}
2\lambda^{2}(h_{1111}^{2}+h_{1112}^{2})-2(a-b)\lambda h_{1111}&\geq -\frac{(a-b)^{2}}{2}
\end{split}
\end{equation*}
 we obtain
\begin{equation}
-\frac{29}{6}a^{2}-\frac{251}{18}ab-\frac{113}{18}b^{2}+S(3a+\frac{13}{4}b)\geq 0.
\end{equation}
Since
$$
\begin{aligned}
&-\frac{29}{6}a^{2}-\frac{58}{18}ab-\frac{13}{2}ab-\frac{13}{3}b^2\\
&=-\frac{29}{36}a(6a+4b)-\dfrac{13}{12}(4b +6a)b=-\frac{29}{36}S(S+3)a-\frac{13}{12}S(S+3)b,
\end{aligned}
$$
we have from (3.4)
$$
\bigl(\frac{21}{36}-\frac{29}{36}S\bigl)Sa  -\frac{76}{18}ab-\frac{35}{18}b^{2}-\frac{13}{12}S^2b\geq 0.
$$

Hence we have $S\leq \dfrac{21}{29}$

\section{Proof of  theorem 1.1}
In this section, we will give a  proof of the theorem 1.1.

\vskip3mm
\noindent
{\it Proof of theorem 1.1}.
We choose a local frame field $\{e_{1},e_{2},e_{3},e_{4}\}$ such that at any point $p$,
\begin{equation*}
h_{ij}=\lambda_{i}\delta_{ij}.
\end{equation*}
Since S is constant, we notice that the sectional curvature  is bounded from below from Gauss equations and $f_{3}$ is bounded.
Hence by making using of the  Generalized Maximum Principle due to Omori \cite{o},
 there exists a sequence of $\{p_{k}\}_{k=1}^{\infty}\subset M^{3}$ such that
\begin{equation*}
\lim_{k \to \infty}f_{3}(p_{k})=\underset{M^{3}}{\sup}\ f_{3},\
\lim_{k \to \infty}|\nabla f_{3}(p_{k})|=0,\
\lim_{k \to \infty}\sup\nabla_l\nabla_l f_{3}(p_{k})\leq 0, \text{\rm for} \ l=1, 2, 3.
\end{equation*}
Since S is constant,
\begin{equation*}
\begin{split}
&\sum_{i,j,k}h_{ijk}^{2}=S(S+3),\\
&\sum_{i,j,k,l}h_{ijkl}^{2}=S(S+3)(S+9)+3(A-2B),
\end{split}
\end{equation*}
we know that , for any $i ,j ,k ,l$,   $\{\lambda_{i}(p_{k})\}$,  \ $\{h_{ijk}(p_{k})\}$ and $\{h_{ijkl}(p_{k})\}$ are bounded sequences respectively.
Thus, we can assume, if necessary,  by taking  a subsequences of $\{p_{m}\}$,
\begin{equation*}
\lim_{m \to \infty} \lambda_{i}(p_{{m}})=\hat{\lambda}_{i},\
\lim_{m \to \infty} h_{ijk}(p_{{m}})=\hat{h}_{ijk},\
\lim_{m \to \infty} h_{ijkl}(p_{{m}})=\hat{h}_{ijkl},\ \ \forall i,j,k,l.
\end{equation*}
From now on, all the computations are considered for $\hat{\lambda}_{i}, \hat{h}_{ijk}$ and $\hat{h}_{ijkl}$. For simple, we omit $\hat{}$.\par

\subsection*{Case 1.  the number of distinct principal curvatures is two}
By the same proof as in the section 3, we get
$$S\leq \dfrac{21}{29}.
$$
\subsection*{Case 2. all three principal curvatures are distinct}
If $f_{3}$ is  constant, then $M^{3}$ is  isoparametric and $S=0$. This is impossible.
From now on, we suppose $f_{3}$ is not  constant.  We will derive a contradiction.\par
We suppose that $\lambda_{1}<\lambda_{2}<\lambda_{3}$. We assume $\sup f_{3}\neq 0$,
otherwise we use $\inf f_{3}\neq 0$.
\begin{lemma}
\label{4.1.}
\begin{equation*}
\begin{split}
&h_{iik}=0,\mbox{ for any $i,k$}.\\
&h_{123}^{2}=\frac{S(S+3)}{6}.
\end{split}
\end{equation*}
\end{lemma}
\begin{proof}
Since $\sum_{i}h_{ii}=0$ and $S=\sum_{i,j}h_{ij}^{2}$ is constant, we have
\begin{equation*}
\sum_{i}h_{iik}=0,\ \sum_{i}h_{iik}\lambda_{i}=0.
\end{equation*}
Since $\lim_{k \to \infty}|\nabla f_{3}(p_{k})|=0$, we have
\begin{equation*}
\sum_{i}h_{iik}\lambda_{i}^{2}=0.
\end{equation*}
Since $\lambda_{i}\neq\lambda_{j}$ for $i\neq j$, we have
$h_{iik}=0$ for any $i,k$. From
$$S(S+3)=\sum_{i,j,k}h_{ijk}^{2}=6h_{123}^{2},
$$
we have
$$h_{123}^{2}=\frac{S(S+3)}{6}.
$$
\end{proof}
\begin{lemma}
\label{4.2.}
\begin{equation*}
h_{iijk}=h_{iiik}=h_{kiii}=0, \ \mbox{for}\  i\neq j\neq k.
\end{equation*}
\end{lemma}
\begin{proof}
From $\sum_{i,j,k}h_{ijk}^{2}=S(S+3)$, we have $h_{123l}=0$ for any $l$. i.e.
\begin{equation}
\label{4.1}
h_{iijk}=0,\ \mbox{for}\ i\neq j\neq k.
\end{equation}
Since $\sum_{i}h_{ii}=0$ and $S=\sum_{i,j}h_{ij}^{2}$ is constant, we have
\begin{equation*}
\sum_{i}h_{iijk}=0\ ,\ \sum_{i}h_{iijk}\lambda_{i}=0.
\end{equation*}
For $j\neq k$,  using the equation $(\ref{4.1})$, we have
\begin{equation}
\label{4.2}
h_{jjjk}+h_{kkjk}=0\ ,\ \sum_{i}h_{iijk}\lambda_{i}=0\ ,\ \mbox{for}\ j\neq k
\end{equation}
From equation $(\ref{4.2})$, we have $h_{jjjk}=h_{kkjk}=0$ for $j\neq k$.
\end{proof}
\begin{lemma}
\label{4.3.}
\begin{equation*}
\sum_{i,k}h_{iikk}^{2}+2\sum_{i\neq k}h_{iikk}^{2}=3S(S+3)^{2}.
\end{equation*}
\begin{proof}
\begin{equation*}
\sum_{i,j,k,l}h_{ijkl}^{2}=S(S+3)(S+9)+3(A-2B),
\end{equation*}
where,
\begin{equation*}
\begin{split}
3(A-2B)&=6h_{123}^{2}(\lambda_{1}^{2}+\lambda_{2}^{2}+\lambda_{3}^{2}-2\lambda_{1}\lambda_{2}-2\lambda_{2}\lambda_{3}-2\lambda_{3}\lambda_{1})\\
&=2S^{2}(S+3),
\end{split}
\end{equation*}
\begin{equation*}
\begin{split}
\sum_{i,j,k,l}h_{ijkl}^{2}&=\underset{l}{\sum_{i\neq j\neq k}}h_{ijkl}^{2}+3\underset{l}{\sum_{i\neq k}}h_{iikl}^{2}+\sum_{i,l}h_{iiil}^{2}\\
&=3\sum_{i\neq k}h_{iikk}^{2}+\sum_{i}h_{iiii}^{2}\\
&=\sum_{i,k}h_{iikk}^{2}+2\sum_{i\neq k}h_{iikk}^{2}.
\end{split}
\end{equation*}
Thus,
$$
\sum_{i,k}h_{iikk}^2+2\sum_{i\neq k}h_{iikk}^2=S(S+3)(S+9)+2S^2(S+3)=3S(S+3)^2.
$$
This proves Lemma $\ref{4.3.}$.
\end{proof}
\end{lemma}
\begin{lemma}
\label{4.4.}
\begin{equation*}
f_{3}> 0,\
-\sqrt{\frac{S}{2}}<\lambda_{1}<-\sqrt{\frac{S}{6}},\
-\sqrt{\frac{S}{6}}<\lambda_{2}<0.
\end{equation*}
\end{lemma}
\begin{proof}
Since
$\lim_{k \to \infty}sup\ \Delta f_{3}(p_{k})\leq 0$
and
$\frac{1}{3}\Delta f_{3}=-(S+3)f_{3}$,
we have
\begin{equation*}
\begin{split}
0&\geq -(S+3)\lim_{k \to \infty}\sup f_{3}(p_{k})\\
&=-(S+3)\underset{M^{3}}{\sup}\ f_{3}\\
&=-(S+3)f_{3}
\end{split}
\end{equation*}
Thus we have $f_{3}>0$.
We also notice that $\lambda_{1}<0$ and $\lambda_{3}>0.$ By direct computation, we have
\begin{equation*}
\begin{split}
f_{3}&=\lambda_{1}^{3}+\lambda_{2}^{3}+\lambda_{3}^{3}=3\lambda_{1}\lambda_{2}\lambda_{3}
=3\lambda_{i}\bigg(\lambda_{i}^{2}-\frac{1}{2}S\bigg),\ \ \ \ \forall i.
\end{split}
\end{equation*}
Hence we have $\lambda_{1}^{2}<\frac{S}{2}$ and $\lambda_{3}^{2}>\frac{S}{2}.$
Since $\sum_{i}\lambda_{i}^{2}=S$ and $\lambda_{3}^{2}>\frac{S}{2}$, we have
\begin{equation}
\label{4.3}
\lambda_{1}^{2}+\lambda_{2}^{2}<\frac{S}{2}.
\end{equation}
Since $\sum_{i}\lambda_{i}^{2}=S$ and $(\ref{4.3})$, we have
\begin{equation*}
\lambda_{1}\lambda_{2}>0.
\end{equation*}
Hence we have $\lambda_{2}<0$. Since $\lambda_{1}^{2}+\lambda_{2}^{2}+\lambda_{1}\lambda_{2}=\frac{S}{2}$ and $\lambda_{1}<\lambda_{2}<0$, we have
\begin{equation*}
\lambda_{1}^{2}>\frac{S}{6},\ \lambda_{2}^{2}<\frac{S}{6}.
\end{equation*}
Hence we have $\frac{S}{6}<\lambda_{1}^{2}<\frac{S}{2},\ 0<\lambda_{2}^{2}<\frac{S}{6}$, and $\lambda_{1}<\lambda_{2}<0$.
This proves our Lemma.
\end{proof}
\begin{lemma}
\label{4.5.}
\begin{equation*}
h_{iikk}=-\frac{1}{3}(S+3)\lambda_{i}+g_{i}\lambda_{k}+wg_{i}g_{k},
\end{equation*}
where
\begin{equation*}
g_{i}=\lambda_{i}^{2}-\frac{f_3}{S}\lambda_{i}-\frac{S}{3}.
\end{equation*}
\end{lemma}
\begin{proof}
Taking the derivative of $\sum_{i}h_{ii}=0$ and $\sum_{i,j}h_{ij}^{2}=S$, we have
\begin{equation*}
\sum_{i}h_{iikk}=0\ ,\ \sum_{i}h_{iikk}\lambda_{i}=-\frac{S(S+3)}{3}.
\end{equation*}
We solve this rank 5 linear system of six equations and six unknowns $h_{iikk}, i\leq k$ with
$h_{iijj}=h_{jjii}+(\lambda_{i}-\lambda_{j})(-1+\lambda_{i}\lambda_{j})$.
\end{proof}

\begin{lemma}\label{4.6.}
$$
f_5=\dfrac56Sf_3, \  f_6=\dfrac{f_3^2}3+\dfrac{S^3}4,$$
$$
 \sum_ig_i^2=\sum_ig_i\lambda_i^2=\dfrac{S^2}6-\dfrac{f_3^2}S, \  \sum_ig_i^4=\dfrac12(\dfrac{S^2}6-\dfrac{f_3^2}S)^2,
 $$
 $$
 \sum_ig_i^2\lambda_i=\dfrac{f_3^3}{S^2}-\dfrac{Sf_3}{6}, \ \sum_ig_i^2\lambda_i^2=\dfrac{S^3}{36}-\dfrac{f_3^2}{6},
\ \sum_ig_i^3\lambda_i=0.
$$
\end{lemma}
\begin{proof}
From  $f_3=3\lambda_i(\lambda_i^2-\dfrac S2)$, for $i=1, 2, 3$,  we have
$$
f_5=\dfrac56Sf_3, \  f_6=\dfrac13f_3^2+\dfrac14S^3.
$$
From $g_{i}=\lambda_{i}^{2}-\frac{f_3}{S}\lambda_{i}-\frac{S}{3}$, we infer
$$
\sum_ig_i^2=\dfrac16S^2-\dfrac{f_3^2}S, \ \ \sum_ig_i^4=\dfrac12(\dfrac16S^2-\dfrac{f_3^2}S)^2,
$$
$$
 \sum_ig_i^2\lambda_i=\dfrac{f_3^3}{S^2}-\dfrac{Sf_3}{6},  \ \sum_ig_i^2\lambda_i^2=\dfrac{S^3}{36}-\dfrac{f_3^2}{6},
 \ \sum_ig_i\lambda_i^2=\dfrac16S^2-\dfrac{f_3^2}S,
$$
According to $F_3=3g_i(g_i^2-\dfrac{F_2}{2})$, for $i=1, 2, 3$, we have
$$
\ \sum_ig_i^3\lambda_i=0,
$$
where $F_k=\sum_ig_i^k$.
\end{proof}

\begin{lemma}
\label{4.7.}
\begin{equation*}
y=\bigg(\frac{1}{3}+\frac{1}{S}\bigg)f_{3}\pm{\bigg[\frac{f_{3}^{2}}{S^{2}}\bigg(\frac{19}{9}S^{2}+\frac{8}{3}S+1\bigg)+\frac{7}{9}S(S+6)\bigg(S+\frac{15}{7}\bigg)\bigg]^{\frac{1}{2}}},
\end{equation*}
where
\begin{equation*}
y=\bigg(\frac{S^{2}}6-\frac{f_{3}^{2}}S\bigg)w.
\end{equation*}
\end{lemma}
\begin{proof}
By using the Lemma $\ref{4.5.}$ and $\ref{4.6.}$, we have
\begin{equation}
\label{4.4}
\begin{aligned}
&\sum_{i,k}h_{iikk}^{2}=\sum_{i,k}\bigl(-\frac{1}{3}(S+3)\lambda_{i}+g_{i}\lambda_{k}+wg_{i}g_{k}\bigl)^2\\
&=\frac{1}{3}S(S+3)^{2}+S\bigg(\frac{S^{2}}6-\frac{f_{3}^{2}}S\bigg)+w^{2}\bigg(\frac{S^{2}}6-\frac{f_3^{2}}S\bigg)^{2},
\end{aligned}
\end{equation}

\begin{equation}
\begin{aligned}
&\sum_{i}h_{iiii}^{2}=\sum_i(-\frac{1}{3}(S+3)\lambda_{i}+g_{i}\lambda_{i}+wg_{i}^2)^2\\
&=\frac{1}{9}(S+3)^2S+\sum_ig_{i}^2\lambda_{i}^2+w^2\sum_ig_{i}^4+2w\sum_i\lambda_ig_i^3\\
&-\dfrac23(S+3)\sum_ig_i\lambda_i^2
-\dfrac23(S+3)w\sum_ig_i^2\lambda_i\\
&=\frac{1}{9}S(S+3)^{2}+\dfrac{S}6\bigg(\frac{S^{2}}6-\frac{f_{3}^{2}}S\bigg)+w^{2}\dfrac12\bigg(\frac{S^{2}}6-\dfrac{f^{2}}S\bigg)^{2}\\
&-\dfrac23(S+3)\bigg(\frac{S^{2}}6-\frac{f_{3}^{2}}S\bigg)-\dfrac23(S+3)w\bigg(\dfrac{f_3^3}{S^2}-\dfrac{Sf_3}{6}\bigg),
\end{aligned}
\end{equation}

\begin{equation}
\label{4.5}
\begin{split}
\sum_{i\neq k}h_{iikk}^{2}&=\sum_{i,k}h_{iikk}^{2}-\sum_{i}h_{iiii}^{2}\\
&
=\frac{2}{9}S(S+3)^{2}+\bigg[\frac{5}{36}S^{3}-\frac{5}{6}f_{3}^{2}\bigg]+\frac{2}{3}(S+3)\bigg(\frac{1}{6}S^{2}-\frac{f_{3}^{2}}{S}\bigg)\\
&+\frac{2}{3}w(S+3)\bigg(-\frac{1}{6}Sf_{3}+\frac{f_{3}^{3}}{S^{2}}\bigg)+\frac{1}{2}w^{2}\bigg[\frac{1}{36}S^{4}-\frac{1}{3}Sf_{3}^{2}+\frac{f_{3}^{4}}{S^{2}}\bigg].
\end{split}
\end{equation}
By substituting equation $(\ref{4.4})$ and $(\ref{4.5})$ into Lemma $\ref{4.3.}$, we prove this Lemma.
\end{proof}
\begin{lemma}
\label{4.7.}
\begin{equation*}
-\frac{3}{S}y\bigg(\lambda_{l}^{2}-\frac{S}{6}\bigg)\bigg(\lambda_{l}^{2}-\frac{2S}{3}\bigg)\leq\lambda_{l}\bigg(\lambda_{l}^{2}-\frac{S}{6}\bigg)\bigg(\frac{9}{S}\lambda_{l}^{4}-\frac{15}{2}\lambda_{l}^{2}+2S+3\bigg).
\end{equation*}
\end{lemma}
\begin{proof}
Since
\begin{equation*}
\frac{1}{3}(f_{3})_{ll}
=\sum_{i}h_{iill}\lambda_{i}^{2}+2\sum_{i,j}h_{ijl}^{2}\lambda_{i},
\end{equation*}
we have
\begin{equation}
\label{4.6}
\begin{split}
0&\geq \frac{1}{3}\lim_{k \to \infty}sup\ (f_{3})_{ll}\\
&=\frac{1}{3}\lim_{k \to \infty}(f_{3})_{ll}\\
&=\sum_{i}h_{iill}\lambda_{i}^{2}+2\sum_{i,j}h_{ijl}^{2}\lambda_{i}.
\end{split}
\end{equation}
By direct computation, we have
\begin{equation}
\label{4.7}
\sum_{i}h_{iill}\lambda_{i}^{2}=-\lambda_{l}(S+3)\bigg(\lambda_{l}^{2}-\frac{1}{2}S\bigg)+\lambda_{l}\bigg(\frac{1}{6}S^{2}-\frac{f_{3}^{2}}{S}\bigg)+\bigg(\lambda_{l}^{2}-\frac{f_{3}}{S}\lambda_{l}-\frac{S}{3}\bigg)y,
\end{equation}
\begin{equation}
\label{4.8}
2\sum_{i,j}h_{ijl}^{2}\lambda_{i}=-\frac{S(S+3)}{3}\lambda_{l}.
\end{equation}
By substituting equation $(\ref{4.7})$ and $(\ref{4.8})$ into $(\ref{4.6})$, we have
\begin{equation*}
\bigg(\lambda_{l}^{2}-\frac{f_{3}}{S}\lambda_{l}-\frac{S}{3}\bigg)y
\leq\lambda_{l}\bigg[(S+3)\bigg(\lambda_{l}^{2}-\frac{1}{2}S\bigg)-\bigg(\frac{1}{6}S^{2}-\frac{f_{3}^{2}}{S}\bigg)+\frac{S(S+3)}{3}\bigg]
\end{equation*}
By direct computation, we complete this Lemma.
\end{proof}
If $y\geq 0$, by using the Lemma $\ref{4.6.}$, we have
\begin{equation}
\label{4.9}
\begin{split}
y&=\bigg(\frac{1}{3}+\frac{1}{S}\bigg)f_{3}+{\bigg[\frac{f_{3}^{2}}{S^{2}}\bigg(\frac{19}{9}S^{2}+\frac{8}{3}S+1\bigg)+\frac{7}{9}S(S+6)\bigg(S+\frac{15}{7}\bigg)\bigg]^{\frac{1}{2}}},\\
&>\bigg(\frac{1}{3}+\frac{1}{S}\bigg)f_{3}+\bigg[\frac{f_{3}^{2}}{S^{2}}\bigg(\frac{4}{3}S+1\bigg)^{2}\bigg]^{\frac{1}{2}}\\
&=\bigg(\frac{5}{3}+\frac{2}{S}\bigg)f_{3}.
\end{split}
\end{equation}
By substituting $l=1$ and $(\ref{4.9})$ into Lemma $\ref{4.7.}$, we have
\begin{equation}
\label{4.10}
\bigg(\frac{24}{S}+\frac{18}{S^{2}}\bigg)\lambda_{1}^{4}-\bigg(25+\frac{21}{S}\bigg)\lambda_{1}^{2}+7S+9<0.
\end{equation}
We notice that the left hand side of equation $(\ref{4.10})$ is an increasing function. Hence substituting $\lambda_{1}^{2}=\frac{S}{6}$ into $(\ref{4.10})$, we have
\begin{equation*}
S<-\frac{12}{7}.
\end{equation*}
It is a contradiction.\par
If $y<0$, by substituting $l=2$ into Lemma $\ref{4.7.}$, we have
\begin{equation}
\label{4.11}
\frac{3}{S}y\bigg(\lambda_{2}^{2}-\frac{2S}{3}\bigg)+\lambda_{2}\bigg(\frac{9}{S}\lambda_{2}^{4}-\frac{15}{2}\lambda_{2}^{2}+2S+3\bigg)\leq 0.
\end{equation}
By substituting
\begin{equation*}
y=\bigg(\frac{1}{3}+\frac{1}{S}\bigg)f_{3}-{\bigg[\frac{f_{3}^{2}}{S^{2}}\bigg(\frac{19}{9}S^{2}+\frac{8}{3}S+1\bigg)+\frac{7}{9}S(S+6)\bigg(S+\frac{15}{7}\bigg)\bigg]^{\frac{1}{2}}}
\end{equation*}
into $(\ref{4.11})$, we notice that the left hand side
of equation $(\ref{4.11})$ is an increasing function. Hence substituting $\lambda_{2}=-\sqrt{\frac{S}{6}}$ into $(\ref{4.11})$, we have
\begin{equation*}
\mbox{LHS of (\ref{4.11})}>0.
\end{equation*}
It is a contradiction.
Hence we have finished the proof of   theorem 1.1.

\bibliographystyle{amsplain}

\begin{thebibliography}{99}
\bibitem{c}
Chang S. P., On minimal hypersurfaces with constant
scalar curvature in $S^4$, J. Diff. Geom., \textbf{37}(1993),
523-534.

\bibitem{cw}
Cheng, Q. M. and Wan, Q. R.,  Complete hypersurfaces of $\mathbb R^4$ with
constant mean curvature, Monatshefte fur Math.  \textbf{118} (1994), 171-204


\bibitem {cdk} Chern S. S. do Carmo M.  \& Kobayashi S.,
Minimal submanifolds of a sphere with second fundamental form of
constant length, Functioal Analysis and Related Fields,
Springer-Verlag,  Berlin, 1970, pp. 59-75


\bibitem {l}  Lawson H. B. Jr., Local rigidity theorems for
minimal hypersurfaces,  Ann. of  Math., \textbf{89} (1969),
167-179.

\bibitem {o} Omori, H., Isometric immersions of Riemannian manifolds,
 J. Math. Soc. Japan  \textbf{19} (1967), 205-214.



\bibitem {pt}  Peng C. K. \& Terng C. L., Minimal
hypersurfaces of spheres with constant scalar curvature, in ``Seminar on minimal submanifolds'', Princeton Univ. Press,
Princeton, 1983, pp. 179-198.

\bibitem {pt1}  Peng C. K. \& Terng C. L., The scalar curvature of
minimal hypersurfaces in sphere, Math. Ann., \textbf{266}(1983),
105-113.
\bibitem {s} Simons J., Minimal varieties
in Riemannian manifolds,  Ann. of Math., \textbf{ 88}(1968),
62-105.



\bibitem {yc1} Yang H. C. \& Cheng Q.-M., A note on the pinching constant
of minimal hypersurfaces with constant scalar curvature in the unit sphere, Chinese Science Bull., \textbf{36}(1991), 1-6.

\bibitem {yc2} Yang H. C. \& Cheng Q.-M., An estimate of the pinching constant of
minimal hypersurfaces with constant scalar curvature in the unit sphere, Manuscripta Math., \textbf{84}(1994), 89-100.

\bibitem {yc3} Yang H. C. \& Cheng Q.-M., Chern's conjecture on
minimal hypersurfaces, Math. Z., \textbf{227}(1998), 377-390.

\bibitem {y} Yau, S. T., Harmonic functions on complete Riemannian manifolds,
Comm. Pure and Appl. Math., \textbf{28} (1975), 201-228.


\end{thebibliography}

\end{document}